\documentclass[12pt]{article}
\usepackage{preamble}
\usepackage{hyperref}
\hypersetup{colorlinks}

\definecolor{indigo}{RGB}{51,0,102}
\definecolor{brightpurple}{RGB}{102,0,153}
\definecolor{fuchsia}{RGB}{180,51,180}
\definecolor{jolightpurple}{RGB}{188,171,240}

\hypersetup{colorlinks,
linkcolor=brightpurple,
filecolor=brightpurple,
urlcolor=indigo,
citecolor=fuchsia}

\usepackage[margin=1.25in]{geometry}

\newcommand{\pl}{P_\Lambda}
\newcommand{\A}{\mathbb{A}}

\title{Symplectic embeddings of four-dimensional polydisks into balls}
\author{Katherine Christianson\footnote{Partially supported by NSF grants DMS-1206667, DMS-0970108, and a Graduate Research Fellowship.} \ and Jo Nelson\footnote{Supported by NSF grant DMS-1303903 and the Charles Simonyi Endowment at the Institute for Advanced Study.}
        }
\date{}
\begin{document}

\maketitle

\begin{abstract}
In this paper we obtain new obstructions to symplectic embeddings of the four-dimensional polydisk $P(a,1)$ into the ball $B(c)$ for $2\leq a \leq \frac{\sqrt{7}-1}  {\sqrt{7}-2} \approx 2.549$, extending work done by Hind-Lisi and Hutchings.  Schlenk's folding construction permits us to conclude our bound on $c$ is optimal.   Our proof makes use of the combinatorial criterion necessary for one ``convex toric domain'' to symplectically embed into another introduced by Hutchings in \cite{Beyond}.  {We also observe that the computational complexity of this criterion can be reduced from $O(2^n)$ to $O(n^2)$.}
\end{abstract}

\tableofcontents

\section{Introduction}
\subsection{New obstructions to embeddings of four-dimensional polydisks}
In this paper we investigate the question of when one convex toric symplectic four-manifold can be symplectically embedded into another.  In particular, we obtain new sharp obstructions to symplectic embeddings of the four-dimensional polydisk $P(a,1)$ into the ball $B(c)$.   {In addition, we prove that the computational complexity in \cite{Beyond} of obstructing symplectic embeddings of convex toric four manifolds can be reduced.}

Four-dimensional toric manifolds are defined as follows.
\begin{definition}
  Let $\Omega$ be a domain in the first quadrant of $\R^2$. Then, 
  we associate to $\Omega$ a subset $X_\Omega$ of $\C^2$ defined by
  \[X_\Omega = \{(z_1,z_2) \in \C^2\ |\ (\pi|z_1|^2,\pi|z_2|^2) \in \Omega\}.\]
  $X_\Omega$ is a symplectic manifold with symplectic form given
  by the restriction of the standard form on $\C^2$, namely
  \[\omega = dx_1 \wedge dy_1 + dx_2 \wedge dy_2.\]
  We call $X_\Omega$ the \emph{toric domain} associated to $\Omega$.
  Suppose that $\Omega$ is of the form
  \[\Omega = \{(x,y) \in \R^2\ |\ 0 \leq x \leq A, 0 \leq y \leq f(x)\},\]
  where $f: [0,A] \to \R_{\geq 0}$ is a nonincreasing function. If $f$
  is concave, then we say that $X_\Omega$ is a \emph{convex toric domain}.
  If $f$ is convex, then we say that $X_\Omega$ is a \emph{concave toric domain}.
\end{definition}
\begin{example}
  Let $\Omega$ be the triangle in $\R^2$ with vertices $(0,0)$, $(a,0)$, and $(0,b)$
  for any $a,b > 0$. Then, $X_\Omega$ is the 4-dimensional ellipsoid
  \[E(a,b)=\left\{(z_1,z_2) \in \C^2\ \left|\ \frac{\pi|z_1|^2}{a} + \frac{\pi|z_2|^2}{b} \leq 1\right.\right\}.\]
When $a = b$, $X_\Omega$ is the 4-dimensional ball $B(a) = E(a,a)$. The ellipsoid $E(a,b)$ is both a concave and a convex toric domain,
  since $\Omega$ is the region lying beneath the line $f(x) = (-b/a)x + b$ in the first quadrant of $\R^2$.
\end{example}
\begin{example}
  Let $\Omega$ be the rectangle in $\R^2$ with vertices $(0,0)$, $(a,0)$,
  $(0,b)$, and $(a,b)$ for any $a,b > 0$. Then, $X_\Omega$ is the polydisk
  \[P(a,b)=\{(z_1,z_2) \in \C^2\ |\ \pi|z_1|^2 \leq a, \pi|z_2|^2 \leq b\}.\]
The polydisk $P(a,b)$ is a convex toric domain, since
  $\Omega$ is the region lying beneath the constant function $f(x) = b$ on the
  interval $[0, a]$.
\end{example}

In dimension 4, progress has been made on understanding questions concerning symplectic embeddings. In \cite{H11}, Hutchings associates to any symplectic four-manifold $(X, \omega)$ with (contact) boundary a sequence of real numbers,
\[0 = c_0(X) \leq c_1(X) \leq c_2(X) \leq \dots, \]
such that if $X$ symplectically embeds into $X'$, then
\[c_k(X) \leq c_k(X') \mbox{ for all } k.\]
 The $c_k$ are called \emph{ECH capacities} (here ECH stands for
``embedded contact homology," which Hutchings uses to define the capacities).  Work by Choi, Cristofaro-Gardiner, Frenkel, Hutchings, and Ramos \cite{CCGFHR},  computed the ECH capacities of all concave toric domains, yielding sharp obstructions to certain symplectic embeddings of concave toric domains.   Cristofaro-Gardiner \cite{CG14} showed that ECH capacities give sharp obstructions to symplectic embeddings of any concave toric domain into any convex toric domain.  His result generalizes the results of McDuff \cite{dusaellipsoid1}-\cite{dusaellipsoid3} and Frenkel-M\"uller \cite{FM}.

Obstructions via ECH capacities are suboptimal in the case of symplectic
embeddings of a convex toric domain into a concave toric domain. For instance,
 the ECH capacities of polydisks and balls (which Hutchings explicitly
computes in \cite{H11}) imply that there is no symplectic embedding of $P(2,1)$ into $B(c)$ for $c < 2$.  However, a result due to Hind and Lisi \cite{HL15} indicates that $P(2,1)$ does not symplectically embed into $B(c)$ for any $c < 3$.

For this reason, Hutchings studied embedded contact homology in a more refined
way than is used to define the ECH capacities.  As a result, he was able to
give a new combinatorial criterion for obstructing symplectic embeddings,
\cite[Theorem 1.19]{Beyond}, which we will hereafter term \emph{the Hutchings
criterion}. The Hutchings criterion is a somewhat complicated combinatorial
condition; we will defer a full description of it to the next section. Hutchings used this criterion to demonstrate several new bounds on embeddings of
polydisks into balls, ellipsoids, and polydisks. 

Our first result is the following extension of results by Hutchings \cite[Theorem 1.4]{Beyond} and Hind-Lisi \cite{HL15} on symplectic embeddings of polydisks into balls.
\begin{theorem}
  \label{thm:14ext}
Let
  \[2 \leq a \leq \frac{\sqrt{7}-1}{\sqrt{7}-2} = 2.54858\dots .\]
If $P(a,1)$ symplectically embeds
  into $B(c)$ then

  \[c \geq 2+\frac{a}{2}.\]
\end{theorem}

\begin{remark} \em
The bound on $c$ in this theorem is optimal: in \cite[Prop. 4.3.9]{Schlenk},
Schlenk uses ``symplectic folding" to construct a symplectic embedding $P(a,1) \into B(c)$
whenever $a > 2$ and $c > 2 + a/2$. 
\end{remark}

\begin{remark} \em
 Hutchings proved the statement of Theorem~\ref{thm:14ext} for $2 \leq a \leq 2.4$ using the Hutchings criterion and conjectured that the full statement of Theorem~\ref{thm:14ext} could be proven using the Hutchings criterion \cite{HBlog}.  Our proof thus answers this conjecture in the affirmative.
\end{remark}

The proof of Theorem~\ref{thm:14ext} can be found in Section~\ref{sec:14ext}.
In Section \ref{sec:14opt} we discuss how extending these results for larger values of $a$ is unlikely to be achieved via the Hutchings criterion or its improvement \cite[Conj. A.3]{Beyond} established by \cite{K16}. For $a > 4,$ it is known that there are symplectic embeddings of $P(a,1)$ into $B(c)$ for some values with  $c < 2 + a/2$; see \cite[Fig. 7.2]{Schlenk}.


  

Our other result is Theorem \ref{thm:bullet3-pairs}, which pertains to the
technical details of the Hutchings criterion. It yields a combinatorial
simplification of the Hutchings criterion for obstructing
symplectic embeddings. This reduces the amount of computations needed to
verify the existence of obstructions from $O(2^n)$ to $O(n^2)$. We state the
result in Section \ref{sec:modify} after reviewing the necessary background.

\subsection{Review of convex generators}
\label{sec:119review}
We begin by defining the principal combinatorial objects involved in stating the Hutchings criterion. Our exposition closely follows \cite[Section 1.3]{Beyond}.

\begin{definition}
  A \emph{convex integral path} $\Lambda$ is a path in $\R^2$ such that:
  \begin{itemize}
    \item The endpoints of $\Lambda$ are $(0,y(\Lambda))$ and
      $(x(\Lambda),0)$ for some non-negative integers $x(\Lambda)$ and
      $y(\Lambda)$.
    \item The path $\Lambda$ is the graph of a piecewise linear concave function
      $f: [0,x(\Lambda)] \to [0,y(\Lambda)]$ with $f'(0) \leq 0$, possibly
      together with a vertical line segment at the right.
          \item The vertices of $\Lambda$ (i.e.\ the points at which its slope
      changes) are lattice points.
  \end{itemize}
\end{definition}
\begin{definition}
  A \emph{convex generator} is a convex integral path $\Lambda$ such that:
  \begin{itemize}
    \item Each edge of $\Lambda$ (i.e.\ each line segment between two vertices)
      is labelled $e$ or $h$.
    \item Horizontal and vertical edges can only be labelled $e$.
  \end{itemize}
\end{definition}

Because we will work with convex generators frequently, we require a compact
notation for them. For any nonnegative, coprime integers $a$ and $b$ and 
any positive integer $m$, we will denote by $e_{a,b}^m$ an edge of a convex
generator that is labelled $e$ and has displacement vector $(ma,-mb)$.
Similarly, $h_{a,b}$ denotes an edge labelled $h$ that has displacement
vector $(a,-b)$, while $e_{a,b}^{m-1}h_{a,b}$ denotes an edge labelled $h$
that has displacement vector $(ma,-mb)$. Since a convex generator is uniquely
specified by the set of its edges, this notation provides an equivalence
between a convex generator and a commutative formal product of symbols
$e_{a,b}$ and $h_{a,b}$, where no two distinct factors $h_{a,b}$ and $h_{c,d}$
have $a = c$ and $b = d$ and where there are no factors of $h_{1,0}$ or
$h_{0,1}$.

As explained in \cite[\S 6]{Beyond}, the boundary of any convex toric domain can be perturbed so that for its induced contact form and up to large action, the ECH generators correspond to these convex generators.  Before continuing to draw parallels with ECH, we first describe a few useful aspects of convex
generators.
\begin{definition}
  Let $\Lambda_1$ and $\Lambda_2$ be convex generators. Then, we say that $\Lambda_1$
  and $\Lambda_2$ \emph{have no elliptic orbit in common} if, when we write out
  $\Lambda_1$ and $\Lambda_2$ as formal products, no factor of $e_{a,b}$
  appears in both $\Lambda_1$ and $\Lambda_2$. Likewise, we say that
  $\Lambda_1$ and $\Lambda_2$ \emph{have no hyperbolic orbit in common} if, when we  write out $\Lambda_1$ and $\Lambda_2$ as formal products, no factor of
  $h_{a,b}$ appears in both $\Lambda_1$ and $\Lambda_2$.
\end{definition}

If $\Lambda_1$ and $\Lambda_2$ are convex generators with no hyperbolic orbit
in commmon, then we define the \emph{product} $\Lambda_1 \cdot \Lambda_2$ to be
the convex generator obtained by concatenating the formal product
expressions of $\Lambda_1$ and $\Lambda_2$. This product operation is associative whenever it is defined.

There are several combinatorial quantities associated to a convex generator
that will be of interest to us.
\begin{definition}
  Let $\Lambda$ be any convex generator.
  \begin{enumerate}
    \item The quantity $L(\Lambda)$ is the number of \emph{lattice points interior to and on the boundary of} the region bounded by $\Lambda$ and the $x$- and  $y$-axes.
     
    \item The quantity $m(\Lambda)$ is the \emph{total multiplicity} of all the edges of  $\Lambda$, i.e.\ the total exponent of all factors of $e_{a,b}$ and $h_{a,b}$ in the formal product for $\Lambda$. Note that $m(\Lambda)$ is equal to one less than the number of lattice points on the path $\Lambda$. 
    
    \item The quantity $h(\Lambda)$ is the number of edges of $\Lambda$ labelled $h.$
  \end{enumerate}
\end{definition}

Remarkably, one can actually express the ECH index in terms of the above combinatorial data associated to convex generators. 

\begin{definition}
If $\Lambda$ is a convex generator, define the \emph{ECH index} of $\Lambda$ to be
      \[I(\Lambda) = 2(L(\Lambda) - 1) - h(\Lambda).\]
\end{definition}

\begin{definition}
  Let $\Lambda$ be a convex generator, and let $X_\Omega$ be a convex toric
  domain. We define the \emph{symplectic action} of $\Lambda$ with
  respect to $X_\Omega$ by
  \[A_\Omega(\Lambda) = A_{X_\Omega}(\Lambda) = \sum_{\nu \in \Edges(\Lambda)} \vec{\nu} \times p_{\Omega, \nu}.\]
  Here, for any edge $\nu$ of $\Lambda$, $\vec{\nu}$ denotes the displacement
  vector of $\nu$, and $p_{\Omega, \nu}$ denotes any point on the line $\ell$
  parallel to $\vec{\nu}$ and tangent to $\partial\Omega$.  Tangency means that
  $\ell$ touches $\partial\Omega$ and that $\Omega$ lies entirely in one
  closed half plane bounded by $\ell$. Moreover, `$\times$' denotes the
   the determinant of the matrix whose columns are given by the
  two vectors.
\end{definition}
Next, we compute the symplectic action of any convex generator with respect to our favorite toric domains.

\begin{example}\label{polyaction}

\begin{itemize}
\item[]
 \item If $X_\Omega = P(a,b)$ is a polydisk, then for any convex generator $\Lambda$,
  \[A_{P(a,b)}(\Lambda) = bx(\Lambda) + ay(\Lambda).\]
\item  If $X_\Omega = E(a,b)$ is an ellipsoid, then for any convex generator $\Lambda$,
  then $A_{E(a,b)}(\Lambda) = c$, where the line $bx + ay = c$ is tangent to
  $\Lambda$ at some point. 
  \end{itemize}
\end{example}

We have yet another definition, which is essential for computing   ECH capacities combinatorially.  
\begin{definition}
  Let $X_\Omega$ be a convex toric domain. We say that a convex generator $\Lambda$
  with $I(\Lambda) = 2k$ for some integer $k$ is \emph{minimal} for $X_\Omega$ if:
  \begin{itemize}
    \item All edges of $\Lambda$ are labelled $e$.
    \item For any other convex generator $\Lambda'$ with all edges labelled $e$
      such that $I(\Lambda') = 2k$, we have
      \[A_\Omega(\Lambda) < A_\Omega(\Lambda').\]
  \end{itemize}
\end{definition}
The symplectic action of minimal generators is related to ECH capacities as follows.
\begin{remark} \em
By \cite[Prop 5.6]{Beyond} if $I(\Lambda)=2k$ and $\Lambda$ is minimal for $X_\Omega$ then $A_\Omega(\Lambda) = c_k(X_\Omega)$.  
\end{remark}

Our final definition will be key to understanding when one convex toric domain can be symplectically embedded into another convex toric domain.
\begin{definition}\label{cobordismineq}
  Let $X_\Omega$ and $X_{\Omega'}$ be convex toric domains, and let $\Lambda$
  and $\Lambda'$ be convex generators. We write
  $\Lambda \leq_{X_\Omega,X_{\Omega'}} \Lambda'$ or
  $\Lambda \leq_{\Omega,\Omega'} \Lambda'$ if
  \begin{enumerate}[label=(\arabic*)]
    \item $I(\Lambda) = I(\Lambda')$,
    \item $A_\Omega(\Lambda) \leq A_{\Omega'}(\Lambda')$, and
    \item $x(\Lambda) + y(\Lambda) - \frac{h(\Lambda)}{2} \geq x(\Lambda') + y(\Lambda') + m(\Lambda') - 1.$
  \end{enumerate}
\end{definition}
In particular, if $X_\Omega$  symplectically embeds into $X_\Omega'$, then the resulting cobordism between their (perturbed) boundaries implies that $\Lambda \leq_{X_\Omega,X_{\Omega'}} \Lambda'$ is a necessary condition for the existence of an embedded irreducible holomorphic curve with ECH index zero between the ECH generators corresponding to $\Lambda$ and $\Lambda'$. The inequality (3) is what ultimately allowed Hutchings to go ``beyond" ECH capacities in his criterion.  It emerges from the fact that every holomorphic curve must have nonnegative genus \cite[Prop 3.2]{Beyond}.

We now have all the ingredients needed to state the Hutchings criterion and our modification.

\subsection{A modification of the Hutchings criterion}\label{sec:modify}
The statement of the criterion we use to obstruct symplectic embeddings will be very similar to the one given by Hutchings in \cite[Thm 1.19]{Beyond}.  Our modification reduces the amount of computation required to check the criterion.
\begin{theorem}[The Modified Hutchings criterion via Thm 1.19 \cite{Beyond}] \hfill
  \label{thm:119ext}
  \smallskip
  
  \noindent Let $X_\Omega$ and $X_{\Omega'}$ be convex toric domains and $\Lambda'$
  be a minimal generator for $X_{\Omega'}$. Suppose that $X_\Omega$
  symplectically embeds into $X_{\Omega'}$. Then, there exists a convex
  generator $\Lambda$, a nonnegative integer $n$, and factorizations
  $\Lambda' = \Lambda_1' \cdots \Lambda_n'$ and
  $\Lambda = \Lambda_1 \cdots \Lambda_n$ such that:
  \begin{enumerate}[label={\em(\roman*)}]
    \item For all $i$, $\Lambda_i \leq_{\Omega,\Omega'} \Lambda_i'$;
    \item For all $i \neq j$, if $\Lambda_i' \neq \Lambda_j'$ or
      $\Lambda_i \neq \Lambda_j$, then $\Lambda_i$ and $\Lambda_j$ have no
      elliptic orbit in common; and
    \item For all $i \neq j$, we have
      $I(\Lambda_i \cdot \Lambda_j) = I(\Lambda_i' \cdot \Lambda_j').$
  \end{enumerate}
\end{theorem}

\begin{remark} \em The difference between Theorem \ref{thm:119ext} and the original Hutchings criterion \cite[Thm 1.19]{Beyond} is in
the third bullet point, where Hutchings' formulation reads:
\begin{itemize}
  \item[(iii)$^\prime$] {\em If $S$ is any subset of $\{1,\dots,n\}$, then}
    $I\left(\prod_{i \in S}\Lambda_i\right) = I\left(\prod_{i \in S}\Lambda_i'\right)$.
\end{itemize}
\end{remark}

We do not lose any information by replacing  (iii) with (iii)$^\prime$ in the Hucthing criterion because of the following proposition and corollary.  Definitions of the terms appearing in the below proposition as well as the proof can be found in Sections \ref{prelim}-\ref{echindexsec}.
 

\begin{proposition}\label{prop:echindex}
Let  $Z_1,...,Z_n$ be relative homology classes, and assume that $CZ_\tau^I(Z_1+...+Z_n) = CZ_\tau^I(Z_1) +...+CZ_\tau^I(Z_n)$.  Then
\[
I(Z_1 +...+Z_n)= \sum_{i<j}I(Z_i +Z_j) - (n-2) \sum_{i=1}^n I(Z_i).
\]
\end{proposition}
Moreover, the assumption of Proposition \ref{prop:echindex} is satisfied for the special contact form arising on the boundary of convex toric domains, by the discussion in Step 4 of the proof of \cite[Lemma 5.4]{Beyond}.  We thus obtain the following corollary since $I(\Lambda)$ is by definition $I$ of any relative homology class between $\Lambda$ and the empty set.  

\begin{corollary}
  \label{thm:bullet3-pairs}
  Let $\{\Lambda_i'\}_{i = 1}^n$ and $\{\Lambda_i\}_{i = 1}^n$ be two sets of
  convex generators such that the $\Lambda_i'$ have no hyperbolic orbit in common
  and the $\Lambda_i$ have no hyperbolic orbit in common. Suppose that for any
  $1 \leq i \leq n$,
  \[I(\Lambda_i) = I(\Lambda_i'),\]
  and moreover that, for any $i \neq j$,
  \[I(\Lambda_i \cdot \Lambda_j) = I(\Lambda_i' \cdot \Lambda_j').\]
  Then, for any subset $S \subseteq \{1,2,\dots,n\}$,
  \[I\left(\prod_{i \in S} \Lambda_i\right) = I\left(\prod_{i \in S} \Lambda_i'\right).\]
\end{corollary}


The proof of Proposition~\ref{prop:echindex} is given in Section~\ref{indexcalc}.  
We note that while Theorem \ref{thm:119ext} is technically weaker than the original Hutchings criterion, Corollary~\ref{thm:bullet3-pairs}  demonstrates that it is actually equivalent to the original Hutchings criterion, \cite[Thm 1.19]{Beyond}.  Thus if we want to check whether some $\Lambda$ obstructs a certain symplectic embedding, it is enough to check whether the conditions in Theorem~\ref{thm:119ext} can be satisfied.  

\begin{remark} \em
Checking that  (iii)$^\prime$ is satisfied requires comparing two indices of convex generators in $O(2^n)$ different scenarios.  Checking that  (iii) is satisfied requires comparing two indices in $O(n^2)$ different scenarios. This vast reduction in complexity is beneficial in many circumstances.
\end{remark}

\noindent \textbf{Outline of paper.} {Properties of the ECH index of two convex generators including the proof of Proposition \ref{prop:echindex} are given in Section \ref{indexcalc}.} The proof of the main embedding result, Theorem \ref{thm:14ext} is given in Section \ref{sec:14ext}. Appendix \ref{sec:14opt} contains a brief discussion on the difficulties in extending Theorem \ref{thm:14ext}.

\begin{acknowledgements}
  This paper grew out of REUs at Columbia University supervised by Jo Nelson
  and Rob Castellano.  We would like to thank Michael Hutchings for suggesting
  this project and for his comments on an earlier draft of this article. For helpful discussions we thank Michael again in addition to Daniel Cristofaro-Gardiner.   {We greatly appreciated the anonymous referee's comments and suggestions, including the simplified proof of Corollary \ref{thm:bullet3-pairs}.} We would also like to thank Robert
  Lipschitz and Chiu-Chu Melissa Liu for making these REUs
  possible.
\end{acknowledgements}


\section{{Index calculations}}\label{indexcalc}
\label{sec:119ext}
In this section we prove Proposition \ref{prop:echindex}.
A formula for the index of the product of two convex generators is proven
in Section \ref{sec:multIform}. 

\subsection{{Preliminary definitions}}\label{prelim}

Let $Y$ be a closed 3-dimensional manifold with a nondegenerate contact form $\lambda$. Let $\xi=\ker(\lambda)$ denote the associated contact structure, and let $R$ denote the Reeb vector field determined by $\lambda$. A {\em Reeb orbit\/} is a map $\gamma:\R/T\Z\to Y$, for some $T>0$, such that $\gamma'(t)=R(\gamma(t))$.  Let $\varphi_t:Y\to Y$ denote the time $t$ Reeb flow. The derivative of $\varphi_t$ at $\gamma(0)$ restricts to a map
\[
d\varphi_t: (\xi_{\gamma(0)}, d\lambda) \to (\xi_{\gamma(t)}, d\lambda).
\]
The \emph{linearized return map} is the map
\begin{equation}
\label{slm}
P_{\gamma}:=d\varphi_T : (\xi_{\gamma(0)},d\lambda) \to (\xi_{\gamma(0)},d\lambda).
\end{equation}
We say that $\gamma$ is \emph{elliptic} if the eigenvalues of $P_{\gamma}$ are on the unit circle, \emph{positive hyperbolic} if the eigenvalues of $P_{\gamma}$ are positive, and \emph{negative hyperbolic} if the eigenvalues of $P_{\gamma}$ are negative.

An \emph{orbit set} is a finite set of pairs $\alpha=\{(\alpha_i,m_i)\}$, where the $\alpha_i$ are distinct embedded Reeb orbits and the $m_i$ are positive integers.  We call $m_i$ the \emph{multiplicity} of $\alpha_i$ in $\alpha$.  The homology class of the orbit set $\alpha$ is defined by
\[
[\alpha]=\sum_i m_i [\alpha_i] \in H_1(Y).
\]
The orbit set $\alpha$ is \emph{admissible} if $m_i=1$ whenever $\alpha_i$ is positive or negative hyperbolic.

Let $\tau$ be a trivialization of $\xi$ over $\gamma$, namely an isomorphism of symplectic vector bundles
\[
\tau: \gamma^*\xi \stackrel{\simeq}{\longrightarrow} (\R/T\Z) \times \R^2.
\]
With respect to this trivialization, the linearized flow $(d\varphi_t)_{t\in[0,T]}$ induces an arc of symplectic matrices $P:[0,T]\to \mbox{Sp}(2)$ defined by
\[
\label{symparc}
P_{t} = \tau(t) \circ d\phi_t \circ \tau(0)^{-1}.
\]
To each arc of symplectic matrices $\{P_t\}_{t\in[0,T]}$ with $P_0=1$ and $P_T$ nondegenerate, there is an associated Conley Zehnder index $CZ(\{P_t\}_{t\in[0,T]})\in\Z$.  We define the \emph{Conley-Zehnder index} of $\gamma$ with respect to $\tau$ by
\[
CZ_\tau(\gamma) = CZ\left(\{P_t\}_{t\in[0,T]}\right).
\]
This depends only on the homotopy class of the trivialization $\tau$.

\subsection{{The ECH index}}\label{echindexsec}

Let $\alpha=\{(\alpha_i,m_i)\}$ and $\beta=\{(\beta_j,n_j)\}$ be Reeb orbit sets  in the same homology class, $\sum_i[\alpha_i]=\sum_j[\beta_j]=\Gamma\in H_1(M).$  Let  $H_2(Y,\alpha,\beta)$ denote the set of 2-chains $Z$ in $Y$ with $\partial Z = \sum_i m_i\alpha_i - \sum_j n_j\beta_j$, modulo boundaries of 3-chains.  The set $H_2(Y,\alpha,\beta)$ is an affine space over $H_2(Y)$.

Given $Z \in H_2(Y,\alpha,\beta)$, we define the \textbf{ECH index} to be
\[
I(\alpha,\beta,Z) = c_\tau(Z) + Q_\tau(Z) +  \sum_i \sum_{k=1}^{m_i}CZ_\tau(\alpha_i^k) - \sum_j \sum_{k=1}^{n_j}CZ_\tau(\beta_j^k),
\]
where $Q_\tau$ is the relative intersection pairing defined in \cite[\S 3.3]{Hu2} and $c_\tau(Z)$ is the relative first Chern class \cite[\S 3.2]{Hu2} of $\xi$ over $Z$ with respect to $\tau$.  The relative intersection pairing is an analogue of the intersection number $[C] \cdot [C]$ for closed curves $C$. As a shorthand, we define
\[
CZ^I_\tau(\alpha) = \sum_i \sum_{k=1}^{m_i}CZ_\tau(\alpha_i^k).
\]
The ECH index does not depend on the choice of trivialization $\tau$.

We note that the Chern class term is linear in the homology class and the relative intersection term is quadratic. The ``total Conley-Zehnder" index term $CZ_\tau^I$ typically behaves in a complicated way with respect to addition of homology classes.  However, we can conclude for the special contact form arising on the boundary of convex toric domains, that the total Conley-Zehnder index term is linear by the discussion in Step 4 of the proof of \cite[Lemma 5.4]{Beyond}.  The addition operation on homology classes to which we refer is spelled out in \cite[Lem. 3.10]{revisited}. Thus, it is reasonable that one only needs to consider ECH indices of one and two term products.

Next we restate and prove Proposition \ref{prop:echindex}.

\begin{proposition} Let  $Z_1,...,Z_n$ be relative homology classes, and assume that $CZ_\tau^I(Z_1+...+Z_n) = CZ_\tau^I(Z_1) +...+CZ_\tau^I(Z_n)$.  Then
\[
I(Z_1 +...+Z_n)= \sum_{i<j}I(Z_i +Z_j) - (n-2) \sum_{i=1}^n I(Z_i).
\]
\end{proposition}
\begin{proof}
Let $L_\tau$  denote the sum $c_\tau+CZ_\tau^I$, which is linear under our assumptions. 
Then
\begin{eqnarray}
I \left( \sum_{i=1}^n Z_i \right) &=& L_\tau\left( \sum_{i=1}^n Z_i  \right) +  Q_\tau \left( \sum_{i=1}^n Z_i  \right) \\
&=&  \sum_{i=1}^n [ L_\tau \left( Z_i \right)  + Q_\tau \left(Z_i\right)]+2\sum_{i=1}^nQ_\tau \left( Z_i,Z_j \right) \label{a} \\
&=&  \sum_{i=1}^n  I \left( Z_i +Z_j \right) - (n-2) \sum_{i=1}^n\label{b}I \left( Z_i \right).
\end{eqnarray}

The second line (\ref{a}) holds here because of the linearity of $L_\tau$, the quadratic property of $Q_\tau$ by \cite[Eq. 3.11]{revisited}, and the fact that $Q_\tau(Z,\cdot)$ is linear in $\cdot$ by definition. The third line (\ref{b}) holds because the $2Q_\tau(Z_i, Z_j)$ terms coming from the terms in the first sum each appear exactly once, while the terms in the first sum that only depend on $Z_i$ all appear $n-1$ times. 
 \end{proof}

\subsection{The index of the product of two convex generators}
\label{sec:multIform}
While we have already proven Proposition~\ref{thm:bullet3-pairs}, we include the following purely combinatorial description of the index of the product of two convex generators.   We expect this to be useful to the future study of obstructing symplectic embeddings of other convex toric domains into concave toric domains.  Before giving the general formula of the index of the product of two convex generators, we first provide an example to elucidate the combinatorial intuition.

Recall that given a convex generator $\Lambda$, $\mathbb{A}(\Lambda)$ was
defined to be the \emph{area of $P_\Lambda$}. Similarly, if $\nu$ is an
edge $\Lambda$ we define $\mathbb{A}_\nu(\Lambda)$ to be the
\emph{area of the portion of $P_\Lambda$ lying underneath $\nu$}. 
We will also need some additional notation as follows.  For any convex generator $\Lambda$ and any edge
$\nu$ of $\Lambda$, we write $\nu_x$ and $\nu_y$ for the $x$- and
$y$-coordinates of the displacement vector of $\nu$. We also define the \emph {slope} of $\nu$ to be
\[
\mu(\nu)= \frac{\nu_y}{\nu_x}.
\]

\begin{example}
  Let $\Lambda = e_{1,0}^3e_{2,1}e_{1,3}$, and let
  $\Gamma = e_{2,1}e_{0,1}^2$. Using \eqref{eqn:PickI2} along with the
  additivity of $b$ and $h$, we have
  \begin{equation}
    \label{eqn:multIex-I}
    \begin{array}{lcl}
    I(\Lambda\cdot\Gamma) &= & 2\A(\Lambda\cdot\Gamma) + b(\Lambda\cdot\Gamma) - h(\Lambda\cdot\Gamma) \\
    & = & 2\A(\Lambda\cdot\Gamma) + b(\Lambda) + b(\Gamma) - h(\Lambda) - h(\Gamma). \\
    \end{array}
  \end{equation}
  We can compute $\A(\Lambda\cdot\Gamma)$ by summing the area under each of the
  edges of $\Lambda\cdot\Gamma= e_{1,0}^3e_{2,1}^2e_{1,3}e_{0,1}^2$.
  
  For any edge $\nu$ of $\Lambda$, the region underneath $\nu$ in
  $\Lambda\cdot\Gamma$ will be essentially the same shape as the region
  under $\nu$ in $\Lambda$, except that $\nu$ may be higher up (i.e.\ its
  endpoints may have larger $y$-coordinates) in the product
  $\Lambda\cdot\Gamma$. To see this, notice that the $y$-coordinate of the
  lower right endpoint of $\nu$ in $\Lambda$ is
  \[y_\Lambda = \sum_{\stackrel{\sigma \in \Edges(\Lambda)}{\mu(\sigma) < \mu(\nu)}} \sigma_y,\]
  while the $y$-coordinate of the lower right endpoint of $\nu$ in
  $\Lambda\cdot\Gamma$ is
  \[y_{\Lambda\cdot\Gamma} = \sum_{\stackrel{\sigma \in \Edges(\Lambda\cdot\Gamma)}{\mu(\sigma) < \mu(\nu)}} \sigma_y.\]
  Thus, every edge $\sigma$ of $\Gamma$ that is steeper than $\nu$ will
  contribute a term of $\sigma_y$ to $y_{\Lambda\cdot\Gamma}$ which is not
  in $y_\Lambda$, so that the edge $\nu$ in $\Lambda\cdot\Gamma$ will be
  translated upwards by $\sigma_y$ relative to the position of $\nu$ in
  $\Lambda$. This translation is equivalent to taking the region beneath
  $\nu$ in $\Lambda$ and adding a rectangle to the bottom of it. So,
  $A_{\Lambda\cdot\Gamma}(\nu)$ will be equal to $A_\Lambda(\nu)$ plus the
  area of several rectangle added beneath $\nu$. Thinking of area in this way
  allows us to break up the area under each edge in $\Lambda\cdot\Gamma$ into
  individual contributions from different edges, as shown in
  Figure~\ref{fig:multIex}.

  \begin{figure}
    \centering
    \begin{tikzpicture}[scale=0.6]
      \draw[->] (0,0) -- (7,0) node[right] {$x$}
      node[midway, below, black] {$\Lambda = e_{1,0}^3e_{2,1}e_{1,3}$};
      \draw[->] (0,0) -- (0,5) node[above] {$y$};
      \filldraw[fill=orange!30!white,draw=orange!50!black] (0,0) rectangle (3,4)
      node[midway, black!80!white] {$\A_\Lambda(e_{1,0}^3)$};
      \filldraw[fill=orange!30!white,draw=orange!50!black] (3,0) -- (3,4) --
      (5,3) node[midway, above, sloped, black!80!white] {$\A_\Lambda(e_{2,1})$}
      -- (5,0) -- cycle;
      \filldraw[fill=orange!30!white,draw=orange!50!black] (5,0) -- (5,3) -- (6,0)
      node[midway, above, sloped, black!80!white] {$\A_\Lambda(e_{1,3})$} -- cycle;

      \draw[->] (9,0) -- (12,0) node[right] {$x$}
      node[midway, below, black] {$\Gamma = e_{2,1}e_{0,1}^2$};
      \draw[->] (9,0) -- (9,4) node[right] {$y$}; 
      \filldraw[fill=green!20!white,draw=green!50!black] (9,0) -- (9,3) -- (11,2)
      node[pos=.6, below, sloped, font=\scriptsize, black!80!white]
      {$\A_\Gamma(e_{2,1})$} -- (11,0) -- cycle;

      \draw[->] (14,0) -- (23,0) node[right] {$x$}
      node[midway, below, black] {$\Lambda\cdot\Gamma = e_{1,0}^3e_{2,1}^2e_{1,3}e_{0,1}^2$};
      \draw[->] (14,0) -- (14,8) node[right] {$y$};
      \filldraw[fill=orange!30!white, draw=orange!50!black] (14,3) rectangle (17,7)
      node[midway, black!80!white] {$\A_\Lambda(e_{1,0}^3)$};
      \filldraw[fill=green!20!white, draw=green!50!black] (14,2) rectangle (17,3)
      node[midway, black!80!white] {$e_{2,1}$};
      \filldraw[fill=green!20!white, draw=green!50!black] (14,0) rectangle (17,2)
      node[midway, black!80!white] {$e_{0,1}^2$};
      \filldraw[fill=orange!30!white, draw=orange!50!black] (17,3) -- (17,7)
      -- (19,6) node[pos=.5, above, sloped, black!80!white]
      {$\A_\Lambda(e_{2,1})$} -- (19,3) -- cycle;
      \filldraw[fill=green!20!white, draw=green!50!black] (17,2) rectangle (19,3)
      node[midway, black!80!white] {$e_{2,1}$};
      \filldraw[fill=green!20!white, draw=green!50!black] (17,0) rectangle (19,2)
      node[midway, black!80!white] {$e_{0,1}^2$};
      \filldraw[fill=green!20!white, draw=green!50!black] (19,3) -- (19,6)
      -- (21,5) node[pos=.6, below, sloped, black!80!white, font=\scriptsize]
      {$\A_\Gamma(e_{2,1})$} -- (21,3) -- cycle;
      \filldraw[fill=orange!30!white, draw=orange!50!black] (19,0) rectangle (21,3)
      node[midway, black!80!white] {$e_{1,3}$};
      \filldraw[fill=orange!30!white, draw=orange!50!black] (21,2) -- (21,5) -- (22,2)
      node[pos=.5, above, sloped, black!80!white] {$\A_\Lambda(e_{1,3})$} -- cycle;
      \filldraw[fill=green!20!white, draw=green!50!black] (21,0) rectangle (22,2)
      node[midway, black!80!white] {$e_{0,1}^2$};
    \end{tikzpicture}
    \caption{The graph on the right shows $\Lambda\cdot\Gamma$ broken up into
      pieces of area from $\Lambda$ and $\Gamma$ along with rectangles added by
      taking the product. Rectangles that were added by taking the product are
      labelled with the edge that necessitated that rectangle. The graphs on
      the left and center show $\Lambda$ and $\Gamma$ for comparison.}
    \label{fig:multIex}
  \end{figure}
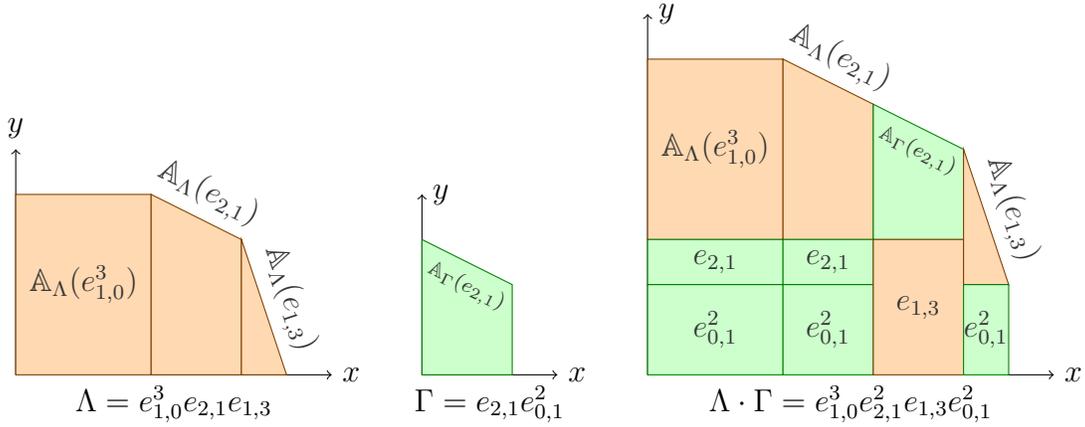

  One important feature of this figure is how we split up the area under the
  edge $e_{2,1}^2$ in $\Lambda\cdot\Gamma$. Because both $\Lambda$ and $\Gamma$
  have an edge of slope $-1/2$, we treat these as separate and compute areas
  underneath them individually, even though they combine to form one edge in
  $\Lambda\cdot\Gamma$. This is important because whichever copy of
  $e_{2,1}$ is on the left (in the figure we've shown it as the one from
  $\Lambda$, but it would not have affected the answer if we'd put the one
  from $\Gamma$ on the left instead) has one rectangle underneath it
  contributed by the other copy of $e_{2,1}$.

  We can now compute $\A(\Lambda\cdot\Gamma)$ by summing up the area
  contributions of each region of $\Lambda\cdot\Gamma$ shown in
  Figure~\ref{fig:multIex}. Let $R$ be the sum of the areas of all the
  rectangles added by taking the product as described above (that is, all
  the rectangles underneath $\Lambda\cdot\Gamma$ in the figure except the
  one labelled $\A_\Lambda(e_{1,0}^3)$). Then,
  \begin{align*}
    \A(\Lambda\cdot\Gamma) &= \A_\Lambda(e_{1,0}^3) + \A_\Lambda(e_{2,1}) + \A_\Gamma(e_{2,1}) + \A_\Lambda(e_{1,3}) + R\\
                        &= \A(\Lambda) + \A(\Gamma) + R.
  \end{align*}
  Plugging back into \eqref{eqn:multIex-I} and applying \eqref{eqn:PickI2}
  then gives
  \begin{align}
  \begin{split}
    \label{eqn:multIex-ans}
    I(\Lambda\cdot\Gamma) &= 2(\A(\Lambda) + \A(\Gamma) + R) + b(\Lambda) + b(\Gamma) - h(\Lambda) - h(\Gamma)\\
                        &= (2\A(\Lambda) + b(\Lambda) - h(\Lambda)) + (2\A(\Gamma) + b(\Gamma) - h(\Gamma)) + 2R\\
                        &= I(\Lambda) + I(\Gamma) + 2R.
  \end{split}
  \end{align}

\end{example}
Equation \eqref{eqn:multIex-ans} is precisely the sort of expression
we want for the index of the product of two convex generators.  By generalizing the above arguments as follows, we obtain a formula for the product of two abitrary generators with no hyperbolic orbit in common, with an explict expression for $R$.

\begin{proposition}
  \label{prop:multIform}
  Let $\Lambda$ and $\Gamma$ be any two convex generators that have no
  hyperbolic orbit in common. Then,
  \[I(\Lambda\cdot\Gamma) = I(\Lambda) + I(\Gamma) + 2\sum_{\nu \in \Edges(\Lambda)}\sum_{\stackrel{\sigma \in \Edges(\Gamma)}{\mu(\sigma) \leq \mu(\nu)}}\nu_x\sigma_y + 2\sum_{\nu \in \Edges(\Gamma)}\sum_{\stackrel{\sigma \in \Edges(\Lambda)}{\mu(\sigma) < \mu(\nu)}}\nu_x\sigma_y.\]
\end{proposition}


\section{On symplectic embeddings of a polydisk into a ball}
\label{sec:14ext}

Our main goal in this section is to prove Theorem~\ref{thm:14ext}, that for $2 \leq a \leq \frac{\sqrt{7}-1}{\sqrt{7}-2}$ and if $P(a,1)$ symplectically embeds into $B(c)$ then $c \geq 2+ a/2$.
Before proceeding we need some preliminary results.  In Section \ref{sec:119ext-setup} we provide some notation and prove a useful formula for the index of a convex generator via Pick's Lemma.  In Section \ref{sec:repcond}  we prove a necessary result regarding the nature of repeated factors in the Hutchings criterion.

With these results in hand, the plan of attack will be to assume that the statement of Theorem~\ref{thm:14ext}
is false and apply the modified Hutchings criterion, Theorem~\ref{thm:119ext}, to the generator $\Lambda' = e_{1,1}^d$
for a suitable choice of $d$.  By \cite[Lemma
2.1]{Beyond} this is a minimal generator for $B(c)$. This gives us an integer $n$, a convex generator
$\Lambda$, and factorizations $\Lambda' = \Lambda_1'\cdots\Lambda_n'$ and
$\Lambda = \Lambda_1\cdots\Lambda_n$. To obtain a contradiction, we
show that no choice of the $\Lambda_i'$ and $\Lambda_i$ is possible.
We do so in three steps.
\begin{enumerate} 
  \item We prove that for sufficiently large $m$, there is no convex generator
    $\Lambda$ such that $\Lambda \leq_{P(a,1),B(c)} e_{1,1}^m$. 
    If we choose $d$ to be very large, this will imply that we cannot have
    $n = 1$. This step is the content of Proposition~\ref{prop:allbd}, which is
    proved in Section~\ref{sec:allbd}.
  \item We use Proposition~\ref{prop:repcond} to show that there cannot exist
    any $i \neq j$ such that $\Lambda_i' = \Lambda_j'$ and
    $\Lambda_i = \Lambda_j$. In conjunction with Step 1, this will imply that
    the set of all possible values of $n$ is bounded. This step is the content
    of Proposition~\ref{prop:noreps}, which is proved in
    Section~\ref{sec:noreps}.
  \item Using Steps 1 and 2, we show that there is a maximum possible index
    of the product $\prod_{i=1}^n \Lambda_i'$ which does not depend on $d$.
    On the other hand, this product must be equal to $\Lambda' = e_{1,1}^d$.
    Because of Step 1, we will be able to pick $d$ to be arbitrarily large,
    which will make the index of $\Lambda'$ arbitrarily large, resulting in
    a contradiction. This step is contained in the proof of
    Theorem~\ref{thm:14ext}, which is given in Section~\ref{sec:14proof}.
\end{enumerate}

After we have proven Theorem~\ref{thm:14ext}, we will discuss whether
it is possible for an application of the Hutchings criterion to extend the
results of the theorem. This discussion is the content of
Section~\ref{sec:14opt}.

\subsection{A helpful lemma via Pick's theorem}
\label{sec:119ext-setup}
We first fix some notation and then prove a useful formula for the index of a convex generator.
For any convex generator $\Lambda$, let $P_\Lambda$ be the region bounded by $\Lambda$ and the $x$- and $y$-axes. 
We define $\mathbb{A}(\Lambda)$ to be the \emph{area of $P_\Lambda$}.

\begin{definition}
For any convex generator $\Lambda$, we define
\[b(\Lambda) = x(\Lambda) + y(\Lambda) + m(\Lambda).\]
\end{definition}

Recall that the formal product 1 is the path $\Lambda$ with no edges which starts and ends at $(0,0)$.   Note that $b(\Lambda)$ computes the lattice points on the boundary of any $\Lambda \neq 1$ if and only if $\Lambda$ does not lie entirely on one axis.
\begin{remark}\em
  The operator $b$ is additive under products of convex generators.
  In other words, for any convex generators $\Lambda$ and $\Gamma$, we have
  \begin{align*}
    b(\Lambda \cdot \Gamma) &= x(\Lambda \cdot \Gamma) + y(\Lambda \cdot \Gamma) + m(\Lambda \cdot \Gamma)\\
                          &= x(\Lambda) + x(\Gamma) + y(\Lambda) + y(\Gamma) + m(\Lambda) + m(\Gamma)\\
                          &= b(\Lambda)+b(\Gamma).
  \end{align*}
\end{remark}

Using the above notation, we can now prove a useful formula for the index
of a convex generator.
\begin{lemma}\label{lem:PickI2}
  Let $\Lambda$ be any convex generator. Then,
  \begin{equation}
    \label{eqn:PickI2}
    I(\Lambda) = 2\A(\Lambda) + b(\Lambda) - h(\Lambda).
  \end{equation}
\end{lemma}
\begin{proof}
  First, suppose that $\Lambda$ lies entirely on one axis.
  If $\Lambda = e_{1,0}^x$ for some $x \geq 0$, we have
  \[I(\Lambda) = 2x = 2 \cdot 0 + 2x - 0 = 2\A(\Lambda) + b(\Lambda) - h(\Lambda).\]
  The case where $\Lambda = e_{0,1}^y$ for some $y \geq 0$ is analogous.

  Next, suppose that $\Lambda$ does not lie entirely on one axis.
 Since $P_\Lambda$ is the region bounded by $\Lambda$ and the $x$- and $y$-axes,   Pick's Theorem yields 
  \[\A(\Lambda) = i(P_\Lambda)+\frac{b(P_\Lambda)}{2}-1,\]
  where $i(P_\Lambda)$ is the number of lattice points in the interior of $P_\Lambda$ and $b(P_\Lambda)$
  is the number of lattice points on the boundary of $P_\Lambda$. 
  Rearranging and noting that $L(\Lambda) = i(P_\Lambda)+b(P_\Lambda)$, we obtain
  \[L(\Lambda) = i(P_\Lambda) + b(P_\Lambda) = \A(\Lambda)+\frac{b(P_\Lambda)}{2}+1 = \A(\Lambda)+\frac{b(\Lambda)}{2}+1,\]
  where the last equality follows from the fact that $\Lambda$ does not lie
  entirely on one axis. We can then use this expression for $L(\Lambda)$ to
  compute $I(\Lambda)$:
  \begin{align}
    \begin{split}
      \label{eqn:PickI}
      I(\Lambda) = 2(L(\Lambda) - 1) - h(\Lambda) 
                                                  &= 2\A(\Lambda) + b(\Lambda) - h(\Lambda).
    \end{split}
\end{align}
\end{proof}

\subsection{Repeated factors in the Hutchings criterion}
\label{sec:repcond}
We will prove the following proposition.
\begin{proposition}
  \label{prop:repcond}
  Let $\Lambda$ and $\Lambda'$ be nontrivial convex generators with no edges
  labelled $h$. Suppose that $\Lambda$ and $\Lambda'$ satisfy (1) and (3) of
  Definition~\ref{cobordismineq} and that $I(\Lambda\cdot\Lambda) = I(\Lambda'\cdot\Lambda')$.
  Then,
  \[8\A(\Lambda') \leq (b(\Lambda')-1)^2.\]
  Moreover, $\Lambda$ must be of the form $e_{x,y}$, where $x, y \in \Z_{>0}$
  are coprime and satisfy
  \[xy = 2\A(\Lambda')\]
  and
  \[x+y = b(\Lambda') - 1.\]
  Equivalently, $x$ and $y$ must be nonnegative coprime integers such that
  \begin{equation}
    \label{eqn:repcond-xyset}
    \{x,y\} = \left\{\frac{b(\Lambda')-1\pm\sqrt{(b(\Lambda')-1)^2-8\A(\Lambda')}}{2}\right\}.
  \end{equation}
\end{proposition}

\begin{proof}
  We will make repeated use of \eqref{eqn:PickI2}, which is the content of
  Lemma~\ref{lem:PickI2}. Using \eqref{eqn:PickI2} along with the additivity of $b$, we get
  \begin{equation}
  \label{eqn:repeat}
\begin{array}{rc l}
    I(\Lambda\cdot\Lambda) &= &2\A(\Lambda\cdot\Lambda) + b(\Lambda\cdot\Lambda) - h(\Lambda\cdot\Lambda)\\
                           &= & 2\A(\Lambda\cdot\Lambda) + 2b(\Lambda).\\
  \end{array}
  \end{equation}
  Recall that $\pl$ denotes the region bounded by $\Lambda$ and the $x$- and
  $y$-axes. Then, the region bounded by $\Lambda\cdot\Lambda$ and
  the $x$- and $y$-axes is $\pl$ dilated by a factor of 2, which has 4
  times the area of $\pl$, i.e.
  \begin{equation}\label{eqn:dumbA}
   \A(\Lambda\cdot\Lambda) = 4\A(\Lambda).
 \end{equation}
 
  Substituting \eqref{eqn:dumbA} into \eqref{eqn:repeat} and using \eqref{eqn:PickI2} again yields
\[
  \begin{array}{rc l}
    I(\Lambda\cdot\Lambda) &= & 8\A(\Lambda) + 2b(\Lambda)\\
                           &= & 4\A(\Lambda) + 2(2\A(\Lambda) + b(\Lambda)) \\ 
                           &=& 4\A(\Lambda) + 2I(\Lambda),\\
  \end{array}
  \]
Likewise for $\Lambda'$ we obtain
  \[I(\Lambda'\cdot\Lambda') = 4\A(\Lambda') + 2I(\Lambda').\]
We assumed that $I(\Lambda\cdot\Lambda) = I(\Lambda'\cdot\Lambda')$, thus,
  \[4\A(\Lambda)+2I(\Lambda) = 4\A(\Lambda')+2I(\Lambda').\]
  Since $I(\Lambda) = I(\Lambda')$, we have
  \begin{equation}
    \label{eqn:repcond-A}
    \A(\Lambda) = \A(\Lambda').
  \end{equation}

  Now, because of \eqref{eqn:PickI2}, we have
  \[I(\Lambda) = 2\A(\Lambda)+b(\Lambda) = I(\Lambda') = 2\A(\Lambda') + b(\Lambda').\] 
  Combining this equation with \eqref{eqn:repcond-A} gives
  \begin{equation}
    \label{eqn:repcond-m1}
    b(\Lambda) = x(\Lambda) + y(\Lambda) + m(\Lambda) = b(\Lambda').
  \end{equation}
  On the other hand, the fact that $\Lambda \leq_{\Omega,\Omega'} \Lambda'$ implies
  that
  \begin{equation}
    \label{eqn:repcond-m2}
    x(\Lambda)+y(\Lambda) \geq b(\Lambda')-1.
  \end{equation}
  Since $m(\Lambda) > 0$, the only way that \eqref{eqn:repcond-m1} and
  \eqref{eqn:repcond-m2} can simultaneously be true is if \eqref{eqn:repcond-m2}
  is an equality and we have $m(\Lambda) = 1$. So, $\Lambda$ must have the form
  $e_{x,y}$, where $\gcd(x,y) = 1$. This allows us to compute properties
  of $\Lambda$ explicitly, so that \eqref{eqn:repcond-m2} becomes
  \begin{equation}
    \label{eqn:repcond-xy1}
    x(\Lambda) + y(\Lambda) = x + y = b(\Lambda') - 1,
  \end{equation}
  and \eqref{eqn:repcond-A} becomes
  \[\A(\Lambda) = \frac{xy}{2} = \A(\Lambda'),\]
  or equivalently
  \begin{equation}
    \label{eqn:repcond-xy2}
    xy = 2\A(\Lambda').
  \end{equation}

  Using \eqref{eqn:repcond-xy1} and \eqref{eqn:repcond-xy2} to solve for
  $x$ and $y$ yields \eqref{eqn:repcond-xyset}.
  Finally, we note that since $x$ and $y$ are real, the square roots in
  \eqref{eqn:repcond-xyset} must be real.
\end{proof}

\begin{remark}\label{whyprophelps} \em
There are a few interesting interactions between the conditions of
Theorem~\ref{thm:119ext} and Proposition~\ref{prop:repcond}.
For instance, Proposition~\ref{prop:repcond} allows us to
rewrite (ii) of Theorem~\ref{thm:119ext} as:
\begin{enumerate}[label=(\roman*), start=2]
  \item For all $i \neq j$, if $\Lambda_i$ and $\Lambda_j$ have any
    elliptic orbit $e_{x,y}$ in common, then $\Lambda_i = \Lambda_j = e_{x,y}$.
\end{enumerate}
In addition, by arguing as in the proof of Theorem~\ref{thm:14ext}, one
can sometimes use (i) of Theorem~\ref{thm:119ext} along with
Proposition~\ref{prop:repcond} to prove that the set of possible
values of $I(\Lambda')$ is bounded. This type
of argument appears in Section~\ref{sec:14proof}.
\end{remark}

\subsection{Elimination of sufficiently large convex generators}
\label{sec:allbd}
We first prove some useful inequalities on the $x$ and $y$ endpoints of certain
convex generators.
\begin{lemma}
  \label{lem:xy-uppers}
  Let $a > 1$ and $c < 2+a/2$, and suppose $d$ and $\Lambda$ are such that
  $\Lambda \leq_{P(a,1),B(c)} e_{1,1}^d$. Then
  \begin{equation}
    \label{eqn:xupper}
    x(\Lambda) < \left(2+\frac{a}{2}\right)d-ay(\Lambda),
  \end{equation}
  and
  \begin{equation}
    \label{eqn:yupper}
    y(\Lambda) < \frac{d(a-2)+2}{2(a-1)}.
  \end{equation}
\end{lemma}
\begin{proof}
By Example \ref{polyaction}, we have
\[
A_{P(a,1)}(\Lambda) =  x(\Lambda) + ay(\Lambda) .
\]
  Our assumptions tell us
  \begin{equation}
    \label{eqn:xyup-acond}
  x(\Lambda) + ay(\Lambda) =A_{P(a,1)}(\Lambda)  \leq A_{B(c)}(e_{1,1}^d) = cd < \left(2+\frac{a}{2}\right)d
  \end{equation}
  and 
  \begin{equation}
    \label{eqn:xyup-hcond}
    x(\Lambda)+y(\Lambda) \geq x(e_{1,1}^d) + y(e_{1,1}^d) + m(e_{1,1}^d) - 1 = 3d-1.
  \end{equation}
  We  solve for $x(\Lambda)$ in \eqref{eqn:xyup-acond} obtaining
  \[x(\Lambda) < \left(2+\frac{a}{2}\right)d-ay(\Lambda).\]
Combining \eqref{eqn:xyup-acond} and \eqref{eqn:xyup-hcond} gives
  \[3d-1+(a-1)y(\Lambda) \leq x(\Lambda)+ay(\Lambda) < \left(2+\frac{a}{2}\right)d.\]
  Solving for $y(\Lambda)$ shows
  \[y(\Lambda) < \frac{d\left(\frac{a}{2}-1\right)+1}{a-1} = \frac{d(a-2)+2}{2(a-1)}.\]
\end{proof}

We now use the above lemma to eliminate sufficiently large convex generators
from consideration in the proof of Theorem~\ref{thm:14ext}.
\begin{proposition}
  \label{prop:allbd}
  Let
  \[2 \leq a < \frac{\sqrt{7}-1}{\sqrt{7}-2},\]
  and suppose that $c < 2+a/2$. Then, there exists some $d_a \geq 1$ such that,
  for any $d > d_a$ and any convex generator $\Lambda$, we have
  $\Lambda \not\leq_{P(a,1),B(c)} e_{1,1}^d$.
\end{proposition}
\begin{proof}
  Fix some $d$, and suppose there exists $\Lambda \leq e_{1,1}^d$. Let
  $x = x(\Lambda)$ and $y = y(\Lambda)$. Because $\Lambda$ is convex, it lies
  inside the rectangle $[0,x] \times [0,y]$. Thus, the maximum possible value
  of $L(\Lambda)$ occurs when $\Lambda$ contains all the lattice points inside
  this rectangle, and the largest $I(\Lambda)$ could be is when $\Lambda$
  contains all these lattice points and has no edges labelled `$h$.' Noting
  also that $I(\Lambda) = I(e_{1,1}^d) = d(d+3)$, we see that
  \[2((x+1)(y+1)-1) = 2(x+1)(y+1)-2 \geq I(\Lambda) = d(d+3),\]
  or equivalently,
  \[0 \geq d(d+3)+2-2(x+1)(y+1).\]
  The substitution of \eqref{eqn:xupper} into this equation yields
  \begin{equation}
    \label{eqn:allbd-I1}
    0 > 2ay^2 -y((4+a)d + 2 - 2a) + d(d+3) - (4+a)d.
  \end{equation}

We now wish to substitute \eqref{eqn:yupper} into \eqref{eqn:allbd-I1}, while still maintaining a valid inequality.  This is permissible provided the right hand side of \eqref{eqn:allbd-I1} is nonincreasing with respect to increasing $y$.
Notice that the derivative of the right hand side of \eqref{eqn:allbd-I1} with respect to $y$ is
  \[4ay-(4+a)d-2+2a.\]
  Substituting \eqref{eqn:yupper} into this expression gives us
  \begin{align}
    \label{eqn:allbd-derivy}
    \begin{split}
    4ay-(4+a)d-2+2a 
                    &< \frac{d(a^2-7a+4) + 2(a^2+1)}{a-1}.
    \end{split}
  \end{align}
  Now, $a^2 - 7a + 4$ has roots at $a = \frac{7\pm\sqrt{33}}{2} \approx 0.628,\ 6.372$.
  Since $a$ is in between these two roots, we have $a^2 - 7a + 4 < 0$.
  So, the expression in \eqref{eqn:allbd-derivy} will be negative for all $d$
  above some sufficiently large value $d_1$. In this case, we can substitute
  \eqref{eqn:yupper} into the right hand side of \eqref{eqn:allbd-I1} and multiply by $2(a-1)^2$ to obtain
  \begin{equation}
    \label{eqn:allbd-dquad}
    0 > (-3a^2+10a-6)d^2 - 2(2a^2+a-1)d + 4(a^2-a+1).
  \end{equation}
  The coefficient of $d^2$ in \eqref{eqn:allbd-dquad} is negative
  for sufficiently large $a$ and has roots at
  $a = \frac{5\pm\sqrt{7}}{3} \approx 0.7848,\ 2.5486$.  Note that $ \frac{5+\sqrt{7}}{3} = \frac{\sqrt{7}-1}{\sqrt{7}-2}.$ Because our value
  of $a$ is between these two roots, we can conclude that the coefficient
  of $d^2$ is positive. Thus, if $d$ is larger than some sufficiently large value
  $d_2$, the right hand side of $\eqref{eqn:allbd-dquad}$ will be positive, a
  contradiction.
  
  We have shown that if $d > d_1$ and $d > d_2$, then the existence of
  $\Lambda$ results in a contradiction. Since $d_1$ and $d_2$ depend only on
  $a$ by construction, setting $d_a = \max\{d_1,d_2\}$ now yields the desired
  statement.
\end{proof}

\subsection{Elimination of repeated factors of convex generators}
\label{sec:noreps}
\begin{proposition}
  \label{prop:noreps}
  Let $2 \leq a \leq 3$, $c < 2+a/2$, and $d \geq 1$. Then, for any
  convex generator $\Lambda$, at least one of the following holds:
 \begin{enumerate}[label={\em(\roman*)}]
\item $\Lambda \not\leq_{P(a,1),B(c)} e_{1,1}^d$.
\item $I(\Lambda\cdot\Lambda) \neq I(e_{1,1}^{2d})$.
\end{enumerate}
\end{proposition}
\begin{proof}
  To obtain a contradiction, suppose that there exists some $\Lambda$ such
  that $\Lambda \leq_{P(a,1),B(c)} e_{1,1}^d$ and
  $I(\Lambda\cdot\Lambda) = I(e_{1,1}^{2d})$. Then, we can apply
  Proposition~\ref{prop:repcond} with $\Lambda' = e_{1,1}^d$. Noting that
  $A(\Lambda') = d^2/2$ and $b(\Lambda') = 3d$, we get
  $\Lambda = e_{x,y}$, where
  \begin{equation}\label{xstuff}
  x = \frac{3d-1 \pm \sqrt{5d^2-6d+1}}{2}
  \end{equation}
  and
  \begin{equation}\label{ystuff}
y = \frac{d^2}{x}.
    \end{equation}

  On the other hand, $\Lambda \leq_{P(a,1),B(c)} e_{1,1}^d$ implies that
  \[  x+ay = x(\Lambda) + ay(\Lambda) = A_{P(a,1)}(\Lambda)   \leq A_{B(c)}(e_{1,1}^d) = cd < \left(2+\frac{a}{2}\right)d.\]
  Substituting in our expression  \eqref{ystuff} for $y$ and multiplying by $x$  gives
  \[x^2+ad^2 < \left(2+\frac{a}{2}\right)xd.\]
  We then substitute in our expression  \eqref{xstuff} for $x$ and multiply by 4 to get
  \begin{multline*}
    (3d-1)^2 \pm (6d-2)\sqrt{5d^2-6d+1}+5d^2-6d+1+4ad^2 <\\ (4+a)d\left(3d-1\pm\sqrt{5d^2-6d+1}\right),
  \end{multline*}
  or equivalently
  \[(2+a)d^2 + (a-8)d + 2 \pm (2d-2-ad)\sqrt{5d^2-6d+1} < 0.\]
  The left hand side of this equation can be factored:
  \begin{equation}
    \label{eqn:noreps-facts}
    \left(-d-1\pm\sqrt{5d^2-6d+1}\right)\left((3-a)d-1\pm\sqrt{1-6d+5d^2}\right) < 0.
  \end{equation}
  The zeros of the left factor (if they exist) occur when
  \[(-d-1)^2 = 5d^2-6d+1,\]
  i.e.\ when $d = 0$ or $d = 2$. Likewise, the zeros of the right factor (if
  they exist) occur when
  \[((3-a)d-1)^2 = 5d^2-6d+1,\]
  or equivalently when
  \[d((-a^2+6a-4)d-2a) = 0.\]
  This equation holds when $d = 0$ and when $d = 2a/(-a^2+6a-4)$. Note that
  for all $2 \leq a \leq 3$, we have $1 \leq 2a/(-a^2+6a-4) < 2$.

  Suppose the sign of the square roots in \eqref{eqn:noreps-facts} is
  positive. Then, both factors in \eqref{eqn:noreps-facts} go to $\infty$ as
  $d \to \infty$, and both possible zeros of both factors are actually zeros
  of these factors. If $d \geq 2$, then $d$ is at least as large as all of
  the zeros of the lefthand side of \eqref{eqn:noreps-facts}, which means that
  the lefthand side of \eqref{eqn:noreps-facts} is nonnegative, a contradiction.
  The only remaining option is $d = 1$. {In this case,  left hand side of \eqref{eqn:noreps-facts} is again nonnegative, a contradiction.}


  Next, suppose the sign of the square roots in \eqref{eqn:noreps-facts} is
  negative. Then, both factors of the lefthand side of \eqref{eqn:noreps-facts}
  go to $-\infty$ as $d \to \infty$, and none of the possible zeros of the
  lefthand side of \eqref{eqn:noreps-facts} is an actual zero. This implies
  that the lefthand side of \eqref{eqn:noreps-facts} is always positive,
  a contradiction.
\end{proof}

\subsection{Proof of Theorem~\ref{thm:14ext}}
\label{sec:14proof}
Throughout this proof, the symbol `$\leq$' between two convex generators
  means `$\leq_{P(a,1),B(c)}$.'
\begin{proof}[Proof of Theorem~\ref{thm:14ext}]

  Suppose by way of contradiction that $c < 2+a/2$ and that $P(a,1)$
  symplectically embeds into $B(c)$. By Proposition~\ref{prop:allbd},
  there exists some $d_a$ such that for any $d > d_a$, there is no convex
  generator $\Lambda$ satisfying $\Lambda \leq e_{1,1}^d$.
  For any $d \in \Z_{>0}$, define
  \[N_d = \#\{\Lambda\ |\ \Lambda \leq e_{1,1}^d\},\]
  and let
  \[N = \sum_{d = 1}^{d_a} dN_d.\]
  Note that for any $d$, there are a finite number of convex generators with
  index equal to $I(e_{1,1}^d)$, which implies that the $N_d$ and $N$ are
  finite.

  Now, fix any integer $D > N$. The generator $\Lambda' = e_{1,1}^D$ is minimal
  for $B(c)$ by \cite[Lemma 2.1]{Beyond}. So, we can apply
  Theorem~\ref{thm:119ext} to obtain a convex generator $\Lambda$, an integer
  $n$, and factorizations $\Lambda' = \Lambda_1'\cdots\Lambda_n'$ and
  $\Lambda = \Lambda_1\cdots\Lambda_n$ satisfying the three numbered conditions
  of Theorem~\ref{thm:119ext}.

  Suppose there exists some $i \neq j$ such that
  $\Lambda_i' = \Lambda_j'$ and $\Lambda_i = \Lambda_j$. Then, let
  $\Gamma = \Lambda_i = \Lambda_j$, and write
  $\Lambda_i' = \Lambda_j' = e_{1,1}^d$ for some $d$. Condition (i) of
  Theorem~\ref{thm:119ext} implies that $\Gamma \leq e_{1,1}^d$, and condition
  (iii) of Theorem~\ref{thm:119ext} implies
  \[I(\Gamma\cdot\Gamma) = I(\Lambda_i'\cdot\Lambda_j') = I(e_{1,1}^{2d}).\]
  However, $\Gamma$ and $d$ then contradict the statement of
  Proposition~\ref{prop:noreps}. So, for all $i \neq j$, we must have
  either $\Lambda_i' \neq \Lambda_j'$ or $\Lambda_i \neq \Lambda_j$. \\

\medskip

{ We claim that with this constraint, it is impossible to have
$I(\Lambda') = I\left(\prod_{i=1}^n\Lambda_i'\right)$:}  \\

\smallskip
{As before, by Proposition \ref{prop:allbd},   there exists some $d_a$ such that for any $d > d_a$, there is no convex generator $\Lambda$ satisfying $\Lambda \leq e_{1,1}^d$. Thus for all $d>d_a$,
  \[
  N_d = \#\{\Lambda\ |\ \Lambda \leq e_{1,1}^d\} =0.
  \]
Assuming $d \leq d_a$, the maximum possible value of $I\left(\prod_{i=1}^n\Lambda_i'\right)$ must then occur when there is precisely one
  choice of $i$ such that $\Lambda_i' = e_{1,1}^d$ and $\Lambda_i = \eta$, for any $\eta \leq e_{1,1}^d$.}
  
{ When $\Lambda_i' = e_{1,1}^d$ and $\Lambda_i = \eta$ we obtain
\[
I\left(\prod_{i=1}^n \Lambda_i'\right) = I\left(\prod_{d=1}^{d_a}\prod_{i=1}^{N_d}e_{1,1}^d\right) = I\left(e_{1,1}^{\sum_{d=1}^{d_a} dN_d}\right) = I(e_{1,1}^N) = N(N+3).\]}

{ If we again fix any integer $D > N$ then the generator $\Lambda' = e_{1,1}^D$ is minimal  for $B(c)$ by \cite[Lemma 2.1]{Beyond}.  Thus
\[
I(\Lambda') = I(e_{1,1}^D) = D(D+3) > N(N+3).
\]
Any other choice of the $\Lambda_i$'s appearing in the factorization of $\Lambda$ must be a subset of the above choice of $\Lambda_i$'s.  As a result, $I\left(\prod_{i=1}^n \Lambda_i'\right)$ will be even smaller. Thus, there are no possible choices for the $\Lambda_i$ such that
$I\left(\prod_{i=1}^n \Lambda_i'\right) = I(\Lambda')$,  contradicting the fact that $I \left( \prod_{I=1}^n \Lambda' \right) = I \left( \Lambda'\right)$.} \\

\bigskip
  
To obtain the statement that if $P(\frac{\sqrt{7}-1}{\sqrt{7}-2},1)$ symplectically embeds into $B(c)$ then $c \geq 2+a/2$ we appeal to the following limiting argument.    Let $a_0=\frac{\sqrt{7}-1}{\sqrt{7}-2}$.  We have just proven, for all $a < a_0$  if $P(a,1)$ symplectically embeds into $B(c)$ then $c \geq 2+\frac{a}{2}.$  Thus if $P(a_0,1)$ symplectically embeds into $B(c)$ then $c \geq 2+\frac{a_0}{2}.$    
\end{proof}


\appendix

\section{Difficulties extending Theorem~\ref{thm:14ext} via the Hutchings criterion}
\label{sec:14opt}

Theorem~\ref{thm:14ext} implies that symplectic folding yields optimal
embeddings of $P(a,1)$ into $B(c)$ whenever
\[2 \leq a \leq \frac{\sqrt{7}-1}{\sqrt{7}-2} = 2.54858\dots\]
For $a > \frac{\sqrt{7}-1}{\sqrt{7}-2}$, our method of proving Theorem~\ref{thm:14ext}
breaks down. More specifically, the proof of Proposition~\ref{prop:allbd} relies
on the fact that $a <  \frac{\sqrt{7}-1}{\sqrt{7}-2}$ in order to conclude that the coefficient of $d^2$ in \eqref{eqn:allbd-dquad} is positive, yielding a contradiction
for sufficiently large $d$. When $a$ is larger than this value, the conclusions
of the proposition will no longer hold, so we will no longer be able to
consider convex generators $e_{1,1}^d$ for arbitrarily large $d$ in the
proof of Theorem~\ref{thm:14ext}.

It is natural to ask whether this upper bound on $a$ can be extended by
applying the Hutchings criterion and using different methods of proof than
those used in Theorem~\ref{thm:14ext}.   

Since $e_{1,1}^d$ is a minimal
generator for $B(c)$ for all $d \geq 1$, we might try applying the Hutchings
criterion to $e_{1,1}^d$ for some specific, not necessarily large choice of
$d$, allowing us to avoid the use of Proposition~\ref{prop:allbd}.
We would then argue as follows. For some fixed $a > \frac{\sqrt{7}-1}{\sqrt{7}-2}$, suppose
we have some $c < 2+a/2$ such that $P(a,1)$ symplectically embeds into $B(c)$.
We can apply the modified Hutchings criterion, Theorem~\ref{thm:119ext}, to $\Lambda' = e_{1,1}^d$ to obtain an
integer $n$, a convex generator $\Lambda$, and factorizations
$\Lambda' = \Lambda_1'\cdots\Lambda_n'$ and $\Lambda = \Lambda_1\cdots\Lambda_n$.

To obstruct the symplectic embedding we assumed to exist, we must show that no
possible choice of the $\Lambda_i$ and $\Lambda_i'$ exists. In particular, we
must show that there exists no convex generator $\Gamma$ such that
$\Gamma \leq_{P(a,1),B(c)} e_{1,1}^d$: otherwise, we will not be able to
obstruct the possibility that $n = 1$, $\Lambda_1' = \Lambda' = e_{1,1}^d$,
and $\Lambda_1 = \Lambda = \Gamma$.

However, we can actually prove that for any $a > \frac{\sqrt{7}-1}{\sqrt{7}-2}$ and
any $d \geq 1$, there is some $c < 2+a/2$ and some convex generator
$\Gamma$ such that $\Gamma \leq_{P(a,1),B(c)}e_{1,1}^d$ for every $d \geq 1$.
This implies that it is impossible to improve on the results of
Theorem~\ref{thm:14ext} by applying the Hutchings criterion to convex
generators of the form $e_{1,1}^d$. The proof of this fact relies on the following construction of a convex generator satisfying certain constraints.

\begin{lemma}
  \label{lem:constructfvm}
  Let $d \geq 9$. Then, there exists some convex generator
  $\Lambda = e_{1,0}^Fe_{m,1}e_{0,1}^V$ such that, 
    \begin{equation}
    \label{eqn:optlt-hbd}
    0 \leq F \leq \frac{1}{2}\left(3d-1+\sqrt{7d^2-3}\right),
  \end{equation}
  \begin{equation}
    \label{eqn:optlt-Vdef}
    V = \frac{1}{2}\left(3d-2-\sqrt{7d^2-6d+4F}\right),
  \end{equation}
  and,
  \begin{equation}
    \label{eqn:optlt-mdef}
    m = \frac{1}{2}\left(3d-2+\sqrt{7d^2-6d+4F}\right) - F.
  \end{equation}
\end{lemma}


\begin{proof}
  Suppose we have some choice of $F$, $V$, and $m$ that satisfies
  \eqref{eqn:optlt-hbd}, \eqref{eqn:optlt-Vdef}, and \eqref{eqn:optlt-mdef}.
  Notice that $V \geq 0$ whenever
  \[(3d-2)^2 \geq 7d^2-6d+4F,\]
  i.e.\ whenever
  \begin{equation}
    \label{eqn:optlt-Vpos}
    \frac{1}{2}(d^2-3d+2) \geq F.
  \end{equation}
  On the other hand, using \eqref{eqn:optlt-hbd} and the fact that
  $d \geq 9$, we have
  \[F \leq \frac{1}{2}(3d-1+\sqrt{7d^2-3}) \leq \frac{1}{2}(d^2-3d+2),\]
  so that \eqref{eqn:optlt-Vpos} is true and $V \geq 0$. Similarly,
  $m \geq 0$ whenever
  \[\sqrt{7d^2-6d+4F} \geq 2F+2-3d,\]
  or equivalently, whenever
  \[F \leq \frac{1}{2}\left(3d-1+\sqrt{7d^2-3}\right).\]
  This inequality is true by \eqref{eqn:optlt-hbd},
  so we must have $m \geq 0$.

  Since $V$ and $m$ are necessarily nonnegative, it remains to find
  some $F$ satisfying \eqref{eqn:optlt-hbd} such that the definitions
  of $V$ and $m$ in \eqref{eqn:optlt-Vdef} and \eqref{eqn:optlt-mdef}
  are integers. Assuming that $\sqrt{7d^2 - 6d+4F}$ is an integer, this
  square root will be even if and only if $d$ is even, which implies that
  $V$ and $m$ will both be integers no matter what the parity of $d$ is.
  So, it suffices to show that we can pick $\sqrt{7d^2-6d+4F}$ to be an
  integer.
  
  Let $k^2$ be the largest perfect square less than $7d^2$. Then, we have
  \begin{equation}
    \label{eqn:optlt-cdef}
    7d^2 - C = k^2
  \end{equation}
  for some $C > 0$. Because $7d^2 < (k+1)^2$, the above equation
  gives us
  \begin{equation}
    \label{eqn:optlt-cbd1}
    C = 7d^2 - k^2 < (k+1)^2 - k^2 = 2k+1 = 2\sqrt{7d^2-C}+1.
  \end{equation}
  The righthand side of this inequality is less than $6d$ whenever
  \[(6d-1)^2 > 28d^2-4C,\]
  or equivalently, whenever
  \begin{equation}
    \label{eqn:optlt-cbd1b}
    8d^2 - 12d + 4C+1 > 0.
  \end{equation}
  The discriminant of the lefthand side of this inequality is
  \[144 - 128C - 32 = 112 - 128C,\]
  which is negative because $C \geq 1$. So,
  \eqref{eqn:optlt-cbd1b} is true, which means that the righthand
  side of \eqref{eqn:optlt-cbd1} is less than $6d$. We then obtain
  \begin{equation}
    \label{eqn:optlt-cbd2}
    C < 6d.
  \end{equation}

  There are now 3 cases to consider, depending on the
  residue class of $C$ modulo 4.
  \begin{enumerate}
    \item Suppose that $C \equiv 1 \bmod{4}$. Taking \eqref{eqn:optlt-cdef}
      mod 4 gives us
      \[-d^2 - 1 \equiv k^2 \bmod{4},\]
      or equivalently,
      \[3 \equiv k^2 + d^2 \bmod{4}.\]
      However, the only squares mod 4 are 0 and 1, so this is impossible.
      Thus, we cannot have $C \equiv 1 \bmod{4}$.
    \item Suppose that $C \equiv 3 \bmod{4}$. Then, we define
      \[F = \frac{1}{2}\left(3d - \frac{C-1}{2} + \sqrt{7d^2-C}\right).\]
      Because $C \geq 3$, this choice of $F$ satisfies the upper bound
      on $F$ given by \eqref{eqn:optlt-hbd}, and \eqref{eqn:optlt-cbd2} gives
      us
      \[F > \frac{1}{2}\left(3d-\frac{C}{2}\right) = \frac{6d-C}{4} > 0,\]
      so that the lower bound of \eqref{eqn:optlt-hbd} is also satisfied.
      Moreover, \eqref{eqn:optlt-cdef} and the fact that $C$ is odd
      imply that $d$ and $\sqrt{7d^2-C} = k$ have opposite parity. So, no
      matter what the parity of $d$ is, $F$ is an integer. Finally, we have
      \[\sqrt{7d^2-6d+4F} = \sqrt{7d^2-C+2\sqrt{7d^2-C}+1} = \sqrt{\left(\sqrt{7d^2-C}+1\right)^2} = k+1,\]
      which is an integer, as desired.
    \item Suppose that $C \equiv 0,2 \bmod{4}$. Then, we pick
      \[F = \frac{6d-C}{4}.\]
      Notice that $F > 0$ by \eqref{eqn:optlt-cbd2}, while
      \[F \leq \frac{1}{2}(3d-1+\sqrt{7d^2-3})\]
      is equivalent to
      \[1-C \leq \sqrt{7d^2-3} = k,\]
      which is true by definition of $C$. Thus, \eqref{eqn:optlt-hbd}
      is satisfied. Moreover, taking \eqref{eqn:optlt-cdef} mod 4 gives
      \[-d^2-C \equiv k^2 \bmod{4},\]
      or equivalently,
      \begin{equation}
        \label{eqn:optlt-cmod}
        -C \equiv k^2+d^2 \bmod{4}.
      \end{equation}
      If $C \equiv 2 \bmod{4}$, then \eqref{eqn:optlt-cmod} implies that
      $k^2 \equiv d^2 \equiv 1 \bmod{4}$, so $d$ must be odd. In this
      case, $6d - C \equiv 0 \bmod{4}$, whence our choice of $F$ is
      an integer. Likewise, if $C \equiv 0 \bmod{4}$, then \eqref{eqn:optlt-cmod}
      implies that $k^2 \equiv d^2 \equiv 0 \bmod{4}$, so $d$
      must be even. In this case, both $6d$ and $C$ are divisible by 4, so
      $F$ is again an integer. Finally, we have
      \[\sqrt{7d^2-6d+4F} = \sqrt{7d^2-C} = k,\]
      which is an integer, as desired.
  \end{enumerate}
\end{proof}

We now use the above lemma to prove that applying the Hutchings criterion to
$e_{1,1}^d$ for any $d \geq 1$ cannot improve upon Theorem~\ref{thm:14ext}.
\begin{proposition}
  \label{thm:optlt}
  Let
  \[a \geq \frac{\sqrt{7}-1}{\sqrt{7}-2} = 2.54858\dots\] 
  For any $d \geq 1$, there exists some $\epsilon > 0$ and some convex
  generator $\Lambda$ such that $\Lambda \leq_{P(a,1),B(c)} e_{1,1}^d$,
  where $c = 2+a/2-\epsilon$.
\end{proposition}

\begin{proof}
  First, note that when $d = 1$, we have $e_{1,0}^2 \leq_{P(a,1),B(c)} e_{1,1}$
  for any $c \geq 2$, and when $d = 2$, we have
  $e_{1,0}^5 \leq_{P(a,1),B(c)} e_{1,1}^2$ for any $c \geq 2.5$. Since 2
  and $2.5$ are less than $2+a/2$ for any possible value of $a$, 
  the desired statement follows for $d = 1,2$. Moreover, if $3 \leq d \leq 8$,
  we can define $\Lambda = e_{1,0}^Fe_{m,1}$ where,
  \[F = \frac{1}{2}(d^2-3d+2),\]
  and,
  \[m = \frac{1}{2}(-d^2+9d-6).\]
  $F$ and $m$ are positive integers for all $3 \leq d \leq 8$. In
  addition we have,
  \begin{equation}
    \label{eqn:optlt-hcond8}
    \begin{array}{lcl}
    x(\Lambda) + y(\Lambda) &=& F+m+1 \\
     &=& \frac{1}{2}(6d-4)+1 \\
     &= &3d-1 \\
     &= &x(e_{1,1}^d) + y(e_{1,1})^d + m(e_{1,1}^d) - 1,\\
     \end{array}
  \end{equation}
  and \eqref{eqn:PickI} yields,
  \begin{equation}
    \label{eqn:optlt-icond8}
        \begin{array}{lcl}
    I(\Lambda) &=& 2F+m+2F+m+2 \\
    & =& 2d^2-6d+4 - d^2 + 9d - 6 + 2 \\
    & =& d^2+3d \\
    &=& I(e_{1,1}^d). \\
    \end{array}
  \end{equation}
  Finally,
  \[A_{P(a,1)}(\Lambda) = x(\Lambda) + ay(\Lambda) = 3d - 2 + a,\]
  so that $A_{P(a,1)}(\Lambda) < (2+a/2)d$ whenever,
  \[a > \frac{2(d-2)}{d-2} = 2.\]
  Because $a > 2$ by assumption, we must have $A_{P(a,1)}(\Lambda) < (2+a/2)d$.
  Then, for any $0 < \epsilon \leq (2+a/2) - A_{P(a,1)}(\Lambda)/d$,
  we obtain
  \[A_{P(a,1)}(\Lambda) \leq (2+a/2-\epsilon)d = A_{B(2+a/2-\epsilon)}(e_{1,1}^d).\]
  This equation along with \eqref{eqn:optlt-hcond8} and \eqref{eqn:optlt-icond8}
  implies that $\Lambda \leq_{P(a,1),B(c)} e_{1,1}^d$ for $c = 2+a/2-\epsilon$,
  as desired.

  We are left with the case where $d \geq 9$. Here, we can apply
  Lemma~\ref{lem:constructfvm} to construct some convex generator
  $\Lambda = e_{1,0}^Fe_{m,1}e_{0,1}^V$ satisfying \eqref{eqn:optlt-hbd},
  \eqref{eqn:optlt-Vdef}, and \eqref{eqn:optlt-mdef}. We will prove
  that $\Lambda \leq_{P(a,1),B(c)} e_{1,1}^d$ for some $c$ of the desired
  form. First, notice that,
  \begin{align}
    \label{eqn:optlt-hcond}
    \begin{split}
    x(\Lambda) + y(\Lambda) &= (F + m) + (V+1)\\
                            &= \frac{1}{2}\left(3d-2-\sqrt{7d^2-6d+4F}\right) + \frac{1}{2}\left(3d-2+\sqrt{7d^2-6d+4F}\right)+1\\
                            &= 3d - 1  \\
                            &= x(e_{1,1}^d) + y(e_{1,1}^d) + m(e_{1,1}^d) - 1.
    \end{split}
  \end{align}
  Moreover, using \eqref{eqn:PickI} and substituting in \eqref{eqn:optlt-hcond}
  gives,
  \begin{align*}
    I(\Lambda) &= 2\A(\Lambda) + x(\Lambda) + y(\Lambda) + m(\Lambda)\\
               &= 2F(V+1)+m(2V+1) + 3d - 1 + F + V + 1\\
               &= 2F(V+1)+2Vm + 3d - 1 + F + V + m + 1\\
               &= 2V(F+m) + 2F + 3d-1 + x(\Lambda) + y(\Lambda)\\
  \end{align*}
  Substituting in \eqref{eqn:optlt-hcond} again and using the definitions of
  $m$ and $V$ produces,
  \begin{align}
    \label{eqn:optlt-icond}
    \begin{split}
    I(\Lambda) &= \frac{1}{2}\left(3d-2-\sqrt{7d^2-6d+4F}\right)\left(3d-2+\sqrt{7d^2-6d+4F}\right) + 2F + 6d-2\\
               &= 2-3d+d^2-2F + 2F + 6d-2\\
               &= d^2 + 3d = I(e_{1,1}^d)
    \end{split}
  \end{align}

  In light of \eqref{eqn:optlt-icond} and \eqref{eqn:optlt-hcond}, we see
  that $\Lambda \leq_{P(a,1),B(c)} e_{1,1}^d$ if and only if
  $A_{P(a,1)}(\Lambda) \leq A_{B(c)}(e_{1,1}^d)$. We will show,
  \begin{equation}
    \label{eqn:optlt-acond}
    A_{P(a,1)}(\Lambda) < \left(2+\frac{a}{2}\right)d.
  \end{equation}
  Then, for any $0 < \epsilon \leq \left(2+a/2\right) - A_{P(a,1)}(\Lambda) / d$,
  we have,
  \[A_{P(a,1)}(\Lambda) \leq \left(2+\frac{a}{2}-\epsilon\right)d = A_{B(2+a/2-\epsilon)}(e_{1,1}^d),\]
  so that $\Lambda \leq_{P(a,1),B(c)} e_{1,1}^d$, where $c = 2+a/2-\epsilon$.
  
  To prove \eqref{eqn:optlt-acond}, we first use
  \eqref{eqn:optlt-hcond} and \eqref{eqn:optlt-Vdef} to compute $A_{P(a,1)}(\Lambda)$:
  \begin{align*}
    A_{P(a,1)}(\Lambda) = x(\Lambda) + ay(\Lambda) &= 3d-1 + (a-1)y(\Lambda)\\
                                                   &= 3d-1 + (a-1)\left(\frac{1}{2}\left(3d-2-\sqrt{7d^2-6d+4F}\right)+1\right)
  \end{align*}
  Using this calculation, \eqref{eqn:optlt-acond} is equivalent to,
  \[(a-1)\left(3d-\sqrt{7d^2-6d+4F}\right) < (a-2)d+2.\]
  Rearranging produces,
  \[\sqrt{7d^2-6d+4F}-d-2 < a\left(\sqrt{7d^2-6d+4F}-2d\right).\]
  Since $\sqrt{7d^2-6d+4F}-2d \geq \sqrt{7d^2-6d}-2d > 0$ for all
  $d > 2$, the above inequality becomes,
  \begin{equation}
    \label{eqn:optlt-acond-equiv}
    \frac{\sqrt{7d^2-6d+4F}-d-2}{\sqrt{7d^2-6d+4F}-2d} < a.
  \end{equation}
  The lefthand side of \eqref{eqn:optlt-acond-equiv} is increasing
  for all $F$ and all $d > 2$, and its limit as $d \to \infty$ is
  \[\frac{\sqrt{7}-1}{\sqrt{7}-2} = 2.54858\dots\]
  Since $a$ is at least this limit value by assumption and $d \geq 9$, we
  conclude that \eqref{eqn:optlt-acond-equiv} is true, hence so is
  \eqref{eqn:optlt-acond}.
\end{proof}

\subsection{Trying other convex generators}
\label{sec:mingens}
Now that we know we cannot use any generator of the form $e_{1,1}^d$ to improve
upon the results of Theorem~\ref{thm:14ext}, we might ask if we can apply the
Hutchings criterion to any other generator for the ball.


First, we investigate other possibilities for minimal generators.  These must \textit{uniquely} minimize the symplectic action among all convex generators of equal index. The following lemma shows that in every index grading other than those of the $e_{1,1}^d$, the action with respect to any ball is non-uniquely minimized, so that the $e_{1,1}^d$ are the only minimal generators for $B(c)$.  

\begin{lemma}
  \label{lem:noaltmins}
  Let $c > 0$, and  let $k$ be a positive integer such that
  $2k \neq I(e_{1,1}^d)$ for all $d \geq 1$. Then, there exist two distinct
  convex generators which minimize the symplectic action with respect to $B(c)$
  among convex generators with index $2k$.
\end{lemma}
\begin{proof}
  The proof is by construction. Let $d$ be the smallest positive integer such
  that $I(e_{1,1}^d) > 2k$, and let $\delta = I(e_{1,1}^d)/2-k$. We construct a finite
  sequence of convex generators $Y_1,Y_2,\dots,Y_\delta$ by induction. In the
  base case, set $Y_1 = e_{1,0}e_{1,1}^{d-1}$. For all $i \geq 2$, define $Y_i$
  from $Y_{i-1}$ according to the following rules.
  \begin{enumerate}
    \item If $Y_{i-1} = e_{1,0}^ae_{1,1}^m$ for some $a$ and $m$, then
      $Y_i = e_{1,0}^{a-1}e_{2,1}e_{1,1}^{m-1}$ if $a > 1$, and
      $Y_i = e_{2,1}e_{1,1}^{m-1}$ if $a = 1$.
    \item If $Y_{i-1} = e_{1,0}^ae_{b,1}e_{1,1}^m$ for some $a$, $b$, and $m$,
      then $Y_i = e_{1,0}^{a-1}e_{b+1,1}e_{1,1}^m$ if $a > 1$, and
      $Y_i = e_{b+1,1}e_{1,1}^m$ if $a = 1$.
    \item If $Y_{i-1} = e_{a,1}e_{1,1}^m$ for some $a$ and $m$, then
      $Y_i = e_{1,0}^{d-m}e_{1,1}^m$.
  \end{enumerate}
  Conceptually, $Y_1$ is equal to $e_{1,1}^d$ but with the uppermost lattice
  point removed, and in general, $Y_i$ is equal to $Y_{i-1}$ with one
  lattice point removed. As an example, the first three $Y_i$ when $d = 3$ are
  shown in Figure~\ref{fig:othermins}.

  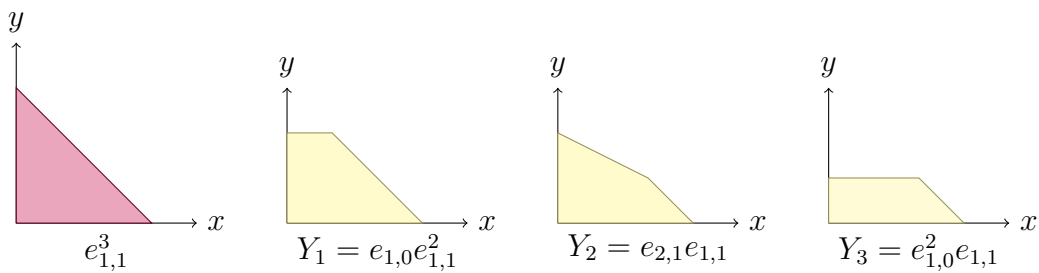
\begin{figure}[htb!]
    \centering
    \begin{tikzpicture}[scale=0.6]
     \draw[->] (-6,0) -- (-2,0) node[right] {$x$}
     node[midway, below, black] {$e_{1,1}^3$};
     \draw[->] (-6,0) -- (-6,4) node[above] {$y$};
     \filldraw[fill=purple!35!white,draw=purple!50!black] (-6,0) -- (-6,3) --
     (-3,0) -- cycle;
     
     \draw[->] (0,0) -- (4,0) node[right] {$x$}
     node[midway, below, black] {$Y_1 = e_{1,0}e_{1,1}^2$};
     \draw[->] (0,0) -- (0,3) node[above] {$y$};
     \filldraw[fill=yellow!25!white,draw=yellow!50!black] (0,0) -- (0,2) --
     (1,2) -- (3,0) -- cycle;

     \draw[->] (6,0) -- (10,0) node[right] {$x$}
     node[midway, below, black] {$Y_2 = e_{2,1}e_{1,1}$};
     \draw[->] (6,0) -- (6,3) node[above] {$y$};
     \filldraw[fill=yellow!25!white,draw=yellow!50!black] (6,0) -- (6,2) --
     (8,1) -- (9,0) -- cycle;

     \draw[->] (12,0) -- (16,0) node[right] {$x$}
     node[midway, below, black] {$Y_3 = e_{1,0}^2e_{1,1}$};
     \draw[->] (12,0) -- (12,3) node[above] {$y$};
     \filldraw[fill=yellow!20!white,draw=yellow!50!black] (12,0) -- (12,1) --
     (14,1) -- (15,0) -- cycle;
    \end{tikzpicture}
    \caption{The first few $Y_i$ when $d = 3$. The generator $e_{1,1}^d$ is
      shown on the left for comparison. Note that every generator is the same
    as the one to the left but with one lattice point removed.}
    \label{fig:othermins}
  \end{figure}

  By construction, we have $I(Y_1) = I(e_{1,1}^d) - 2$ and
  $I(Y_i) = I(Y_{i-1}) - 2$ for all $i \geq 2$. This implies that
  $I(Y_i) = I(e_{1,1}^d) - 2i$ for all $i$ and in particular that
  $I(Y_\delta) = 2k$. We claim that $Y_\delta$ minimizes the
  symplectic action with respect to $B(c)$ among all convex
  generators with index $2k$.
  
  To this end, note that for any convex generator $\Lambda$ with
  $I(\Lambda) = 2k$, we have $A_{B(c)}(\Lambda) = c(m+n)$, where $(m,n)$ is
  the vertex of $\Lambda$ at which a line of slope $-1$ is tangent. Now, $m+n$
  is the $y$-intercept of the line of slope $-1$ through $(m,n)$. For any other
  vertex $(a,b)$ of $\Lambda$, the line of slope $-1$ through $(a,b)$ is not
  tangent to $\Lambda$ and so has strictly smaller $y$-intercept than the
  tangent line of slope $-1$. This implies that $m+n \geq a+b$ for any
  vertex $(a,b)$ of $\Lambda$, with equality if and only if $(m,n) = (a,b)$.

  Now, $I(\Lambda) = 2k > I(e_{1,1}^{d-1})$ by the definition of $d$, so we
  know that $\Lambda$ contains some lattice point $(a,b)$ not contained in
  $e_{1,1}^{d-1}$. Using our above arguments, we then have
  \[A_{B(c)}(\Lambda) = c(m+n) \geq c(a+b) > c(d-1),\]
  so that in fact, $A_{B(c)}(\Lambda) \geq cd$. On the other hand,
  the line $x+y = d$ is tangent to $Y_i$ for all $i$ by construction, which
  implies that $A_{B(c)}(Y_i) = cd$. In particular, we obtain
  \[A_{B(c)}(Y_\delta) = cd \leq A_{B(c)}(\Lambda),\]
  as desired.

  Next, define $X_\delta$ to be the reflection of $Y_\delta$ about the line
  $y = x$. The line $x+y = d$ is tangent to $X_\delta$, so we have
  $A_{B(c)}(X_\delta) = A_{B(c)}(Y_\delta) = cd$. This implies that $X_\delta$
  also minimizes the symplectic action of $B(c)$ among convex generators with
  index $2k$. Finally, we note that $X_\delta \neq Y_\delta$ because $Y_\delta$
  is not symmetric about the line $y = x$.
\end{proof}

As a result, we cannot apply Theorem~\ref{thm:119ext} to any
convex generators other than the $e_{1,1}^d$ in order to understand symplectic
embeddings into the ball. Combined with Theorem~\ref{thm:optlt}, this implies
that in fact, Theorem~\ref{thm:119ext} cannot be used to extend the upper bound
on $a$ in the statement of Theorem~\ref{thm:14ext}.

The improvement of the Hutchings criterion \cite[Conj. A.3]{Beyond},
proven in \cite{K16}, allows the statement of Theorem~\ref{thm:119ext} to be
weakened so that one need only assume that all edges of $\Lambda'$ are labelled
`$e$' (as opposed to the requirement that $\Lambda'$ be minimal). 
As a result, one could conceivably improve upon Theorem~\ref{thm:14ext} using a non-minimal generator.

For instance, we could try to apply the Hutchings
criterion to the convex generators constructed in Lemma~\ref{lem:noaltmins},
which non-uniquely minimize the symplectic action in their index grading.
However, preliminary evidence suggests that these generators (as well as all
others of equal index and symplectic action) will do no better than the
$e_{1,1}^d$.

Moreover, \cite[Conj. A.3]{Beyond} would also allow one to use a generator
that does not minimize the symplectic action at all. This choice would likely
weaken the action inequality in the definition of `$\leq$' between convex
generators for most relevant cases.  Thus the Hutchings criterion should on the
whole yield weaker combinatorial conditions for non-minimal generators than it
does for minimal ones. In short, some possibility remains to extend the
statement of Theorem~\ref{thm:14ext} to larger values of $a$ using the Hutchings
criterion, but it will require methods beyond the scope of this paper.

\noindent \textsc{Katherine Christianson \\  UC Berkeley}\\
{\em email: }\texttt{christianson@math.berkeley.edu}\\
\medskip

\noindent \textsc{Jo Nelson \\  Columbia University}\\
{\em email: }\texttt{nelson@math.columbia.edu}\\

\end{document}